\documentclass[12pt,reqno]{amsart}
\usepackage{geometry, graphicx, url, array}
\usepackage{amsfonts,color,amsmath,amssymb,amsthm}
\usepackage{placeins}

\usepackage{array}
\usepackage{booktabs}

\newcommand{\F}{{\mathbb F}}

\newcommand{\ra}{\rangle}
\newcommand{\la}{\langle}

\newcommand{\cM}{{\mathcal M}}

\newcommand{\PAut}{\mathrm{PAut}}

\newcommand{\Aut}{\mathrm{Aut}}
\newcommand{\sym}{\mathrm{Sym}}

\newcommand{\Sp}{\mathrm{Sp}}

\newcommand{\PG}{\mathrm{PG}}

\newtheorem{theorem}{Theorem}

\newtheorem{lemma}[theorem]{Lemma}
\newtheorem{corollary}[theorem]{Corollary}

\newtheorem{proposition}[theorem]{Proposition}

\newtheorem{example}[theorem]{Example}
\numberwithin{equation}{section}
\numberwithin{theorem}{section}
\numberwithin{table}{section}
\newtheorem{conjecture}[theorem]{Conjecture}
\newtheorem{remark}[theorem]{Remark}

\begin{document}

\title[]{On the lengths of MDS codes with a two-transitive permutation automorphism group}
\author[Deng, Feng, Vasil'ev]{Haihua Deng, Tao Feng, Andrey V. Vasil'ev}
\address{Haihua Deng, School of Mathematical Sciences, Zhejiang University, 866 Yuhangtang Road,  Hangzhou 310058, Zhejiang, China}
\email{haihua.deng@zju.edu.cn}
\address{Tao Feng, School of Mathematical Sciences, Zhejiang University, 866 Yuhangtang Road,  Hangzhou 310058, Zhejiang, China}
\email{tfeng@zju.edu.cn}
\address{Andrey V. Vasil'ev, Sobolev Institute of Mathematics, Novosibirsk, 630000, Russia} \email{vasand@math.nsc.ru}


\begin{abstract}
Let $C$ be an $[n,k]_q$ maximum distance separable (MDS) code with $4\le k\le q-3$, and suppose that it has a $2$-transitive permutation automorphism group. In this paper we show that $n\le q+1$, so the MDS conjecture holds for this class of codes.
\end{abstract}

\maketitle


\section{Introduction}

Let $q$ be a prime power, and let $\F_q$ be the finite field with $q$ elements. An $[n,k,d]_q$ code $C$ is a $k$-dimensional subspace of $\mathbb{F}_q^n$ with minimum Hamming distance $d$. If $d$ is unspecified, we refer to $C$ as an $[n,k]_q$ code. If $2\le k\le n-2$, then we call the code $C$ \textit{nontrivial}. We say that it is a \textit{maximum distance separable} (MDS) code if it attains equality in the Singleton bound, i.e., $d=n-k+1$ \cite{Singleton1964}. The MDS conjecture, which we state below, is a fundamental and long-standing problem in coding theory.
\begin{conjecture}\label{conj_MDS}
Let $C$ be a nontrivial $[n,k]_q$ MDS code. Then $n\le q+1$, except when $q$ is even and $k\in\{3,q-1\}$, in which case $n\le q+2$.
\end{conjecture}

The MDS conjecture has its roots in a statistical problem discussed in 1952 by Bush in \cite{Bush1952}. Segre started the influential geometric approach to this problem in \cite{Segre1955Curve}, which was widely adopted and eventually led to the breakthrough of Ball \cite{Ball2012} in 2012. We give a brief description of the relation with finite geometry as follows. Let $C$ be an $[n,k]_q$ MDS code with $k\ge 2$, and let $\mathrm{PG}(k-1,q)$ be the projective space whose points are the one-dimensional subspaces of $\mathbb{F}_q^k$. Let $M$ be a generator matrix of $C$, and let $\cM$ be the set of $n$ projective points corresponding to the column vectors of $M$. It is well known that any $k$ column vectors of $M$ are linearly independent, cf. \cite{HuffmanPless2003} or \cite{MacWilliamsSloane1977}. It follows that any $k$ points of $\cM$ do not lie on a hyperplane, i.e., $\cM$ is a $n$-arc of $\mathrm{PG}(k-1,q)$. Conversely, every $n$-arc gives rise to an $[n,k]_q$ MDS code, cf. \cite{HirschfeldThas1991}. Segre \cite{Segre1955Curve} associated an algebraic curve to an $n$-arc, and this perspective led to important progress towards the MDS conjecture by using deep tools from algebraic geometry, cf. \cite{BlokhuisBruenThas1988,BlokhuisBruenThas1990,Voloch1991,StormeThas1993,HirschfeldKorchmaros1996,HirschfeldKorchmaros1998}. For an exposition of geometric techniques and recent progress on the MDS conjecture, we refer the reader to the comprehensive survey \cite{BallLavrauw2020}.

The permutation automorphism group $\mathrm{PAut}(C)$ of a code $C$ (see Section \ref{sec:prelim} for definition), referred to as the permutation group of $C$ for short, is an important research topic in coding theory. MacWilliams explored $\mathrm{PAut}(C)$ to develop efficient permutation decoding algorithms \cite{MacWilliams1964}, and the same idea has recently been developed into automorphism ensemble decoding for Reed--Muller codes \cite{GeiselhartEtAl2021}. In 2012 Kaufman and Lubotzky \cite{KaufmanLub} constructed the first family of asymptotically good symmetric LDPC codes, where a code $C$ of length $n$ is symmetric if $\mathrm{PAut}(C)$ is transitive on the $n$ coordinate positions. The extremal self-dual codes over small fields with $2$-transitive permutation groups have been studied in \cite{MalevichWillems2014, ChigiraHaradaKitazume2014}. We refer to \cite{KnappSchmid1980,KnappRodrigues2021,PaceSonnino2017} for group-based constructions of linear codes, \cite{BienertKlopsch2010,FengHollmannLiXiang2026,GuendaGulliver2013,MaYan2025} for recent progress on the permutation groups of cyclic codes, and refer to \cite{Huffman1998} for a detailed description of the relation between codes and groups.

In this paper, we initiate the study of the length of an MDS code $C$ with large transitive permutation group $\mathrm{PAut}(C)$. We list the known MDS codes with a $2$-transitive permutation group as follows; a complete classification of such codes is still unavailable.

\begin{example}\label{ex:two-transitive-mds-codes}
\text{ }
\begin{enumerate}
\item The repetition code $J$ of length $n$ over $\F_q$ spanned by the all-$1$ vector is an $[n,1,n]_q$ code; the code $A=J^{\perp}$ is an $[n,n-1,2]_q$ code, and the whole space $\F_q^n$ is an $[n,n,1]_q$ code. They are all trivial MDS codes with permutation automorphism group $\sym(n)$.
\item Let $\F_r\subseteq\F_q$ and $1\le k\le r$. We define
\[
\mathrm{RS}_{r,k}(\F_q)=\left\{\left(f(x)\right)_{x\in\F_r}: f\in\F_q[X],\ \deg f<k \right\}.
\]
This is an $[r,k,r-k+1]_q$ MDS code, and $\mathrm{AGL}_1(r)\le\mathrm{PAut}(C)$ acts $2$-transitively on the set $\F_r$ of coordinate positions.

\item Let $r=2^h$, $\F_r\subseteq\F_q$, and take $1\le e<h$ and $\gcd(e,h)=1$. Let
\[
\mathcal T_e(r)=\left\langle(1)_{x\in\F_r},
 (x)_{x\in\F_r}, (x^{2^e})_{x\in\F_r}\right\rangle_{\F_q}\le\F_q^r,
\]
and for $r\ge8$ let
\[
 \mathcal B_e(r)=\left\langle (1)_{x\in\F_r},
 (x)_{x\in\F_r}, (x^{2^e})_{x\in\F_r},(x^{2^e+1})_{x\in\F_r}
 \right\rangle_{\F_q}\le\F_q^r.
\]
Then $\mathcal T_e(r)$ and $\mathcal B_e(r)$ are $[r,3,r-2]_q$ and $[r,4,r-3]_q$ MDS codes respectively. The group $\mathrm{AGL}_1(r)$ acts $2$-transitively on the set $\F_r$ of coordinate positions. These codes arise from a translation hyperoval and Segre's $(r+1)$-arc respectively, cf. \cite[Examples~4.1 and~4.2]{BallLavrauw2020}. If $e=2$ and $\gcd(e,h)=1$, Berger showed that the dual of $\mathcal T_2(r)$ is an affine-invariant MDS code not equivalent to Reed--Solomon codes \cite{Berger1993}.

\item Let $\mathcal H_6$ be the $[6,3,4]_4$ hexacode. If $\F_4\subseteq\F_q$, then the $\F_q$-span of all its codewords is a $[6,3,4]_q$ MDS code. Its permutation automorphism group contains $\mathrm{PSL}_2(5)\cong A_5$, which is $2$-transitive on the $6$ coordinate positions.
\end{enumerate}
\end{example}

Our main result in this paper is the following theorem.
\begin{theorem}\label{thm:main}
Let $C$ be an $[n,k]_q$ MDS code with $4\le k\le q-3$. If $\mathrm{PAut}(C)$ is $2$-transitive on the $n$ coordinate positions, then $n\le q+1$.
\end{theorem}

The rest of this paper is organized as follows. In Section~\ref{sec:prelim}, we fix the notation and prove some preliminary lemmas that we shall use later. In Section~\ref{sec:reduction}, we apply Mortimer's results on modular representations of $2$-transitive permutation groups from \cite{mortimer1980modular} and some related results from \cite{GuralnickTiep2011} to reduce the proof of Theorem \ref{thm:main} to the case where $\PAut(C)$ is one of some infinite families of almost simple groups. In Section~\ref{sec:almost-simple} we treat those infinite families and complete the proof of Theorem \ref{thm:main}. We conclude this paper with some related research problems in Section~\ref{sec:concluding}.

\section{Preliminaries}\label{sec:prelim}

Let $C$ be an $[n,k]_q$ code whose coordinates are labeled by an $n$-element set $\Omega$. It is conventional to take $\Omega=\{1,2,\ldots,n\}$, but we will allow $\Omega$ to be other sets. The support of a codeword $c\in C$ is $\mathrm{supp}(c)=\{\omega\in\Omega:c_\omega\ne0\}$, and its Hamming weight is the size of $\mathrm{supp}(c)$. We write $d(C)$ for the smallest nonzero Hamming weight of codewords in $C$.

Let $\F=\F_q$ for the finite field with $q$ elements, where $q=p^f$ with $p$ prime. Let $\F^\Omega$ be the set of all functions $f:\Omega\rightarrow \F$. For each $\omega\in\Omega$, let $e_\omega$ be the function that takes $1$ at $\omega$ and takes $0$ otherwise. For a subset $\Delta\subseteq\Omega$, write $\mathbf{1}_\Delta=\sum_{\omega\in\Delta}e_\omega$, and put $\mathbf{1}=\mathbf{1}_\Omega$. For $f,f'\in\F^\Omega$, their standard inner product is $\la f,f'\ra=\sum_{\omega\in\Omega}f(\omega)f'(\omega)$. We identify $C$ as a subspace of $\F^\Omega$ via
\[
c=(c_\omega)_{\omega\in\Omega}\mapsto \sum_{\omega\in\Omega}c_\omega e_\omega,
\]
and write $c$ for $\sum_{\omega\in\Omega}c_\omega e_\omega$ by abuse of notation. The \textit{dual code} of $C$ is $C^\perp=\{x\in \F^\Omega:\langle x,c\rangle=0\text{ for each }c\in C\}$, which has dimension $n-k$.

Let $\sym(\Omega)$ be the set of all permutations of $\Omega$. It induces an action on $\F^\Omega$ as follows:
\[
(f^g)(\omega)=f(\omega^{g^{-1}}),\;\textup{ for }g\in\sym(\Omega),\,\omega\in\Omega,\, f\in\F^\Omega.
\]
In particular, we have $e_\omega^g=e_{\omega^g}$. This action preserves the standard inner product:
\begin{equation}\label{eqn_presInnProd}
\la f^g,f'^g\ra=\la f,f'\ra \textup{ for } f,f'\in\F^\Omega,\;g\in\sym(\Omega).
\end{equation}
We define the permutation (automorphism) group of the code $C$ as follows:
\[
\mathrm{PAut}(C)=\{g\in\sym(\Omega):c^g\in C\textup{ for each }c\in C\}.
\]
By \eqref{eqn_presInnProd}, we immediately deduce the following result.
\begin{proposition}\label{prop:PAutdual}
  We have $\mathrm{PAut}(C)=\mathrm{PAut}(C^\perp)$.
\end{proposition}
If $\mathrm{PAut}(C)$ acts transitively on the set $\{(\alpha,\beta)\in\Omega^2:\alpha\ne\beta\}$, then we say that $\mathrm{PAut}(C)$ is $2$-transitive on the coordinate set $\Omega$. The $2$-transitive permutation groups have been classified, and the complete list can be found in \cite[Tables~7.3, 7.4]{Cameron1999} or \cite[Section~7.7]{DixonMortimer1996}.

\begin{remark}\label{rem_ElePAut}
Here is a matrix description of $g\in \mathrm{PAut}(C)$. Let $B_g$ be the $n\times n$ permutation matrix such that $B_g(\omega,\omega')=1$ if $\omega'=\omega^g$ and $=0$ otherwise. Let $M$ be a generator matrix of $C$. Since $g$ stabilizes $C$, the row space of $MB_g$ is $C$. There is a corresponding invertible matrix $P_g$ of order $k$ such that $P_gM=MB_g$. We call $B_g$ the permutation matrix associated with $g$, and call $P_g$ the companion matrix with respect to the generator matrix $M$.
\end{remark}

We shall make use of the following well-known duality fact multiple times.
\begin{lemma}\label{lem_duality}
Let $C$ be a nontrivial $[n,k]_q$ code. Then $C$ is MDS if and only if $C^\perp$ is MDS.
\end{lemma}

Let $G=\mathrm{PAut}(C)$. The vector space $\F^\Omega$ is naturally an $\F G$-module, and $C$ is its submodule. Let $E$ be an extension field of $\F$. We write $C\otimes E$ for the $E$-span of $C$, which is a subspace of $E^\Omega$. It has dimension $k$ over $E$, where $k=\dim_{\F}(C)$. Moreover, if $v_1,\cdots,v_k$ form a basis of $C$ over $\F$, then they also form a basis of $C\otimes E=\la v_1,\ldots,v_k\ra_{E}$.

 \begin{lemma}\label{lem_scalar-extension}
 Let $E$ be an extension field of $\F$, and let $C\le\F^\Omega$ be an MDS code. Then $C\otimes E$ is an MDS code over $E$ with the same parameters as $C$, and $\mathrm{PAut}(C)=\mathrm{PAut}(C\otimes E)$.
 \end{lemma}
 \begin{proof}
 Let $C'=C\otimes E$ in this proof. We choose a generator matrix $M$ of $C$ which has entries in $\F$, and it is also a generator matrix of $C\otimes E$. Let $k$ be the dimension of $C$. Any $k$ columns of $C$ are linearly independent over $\F$, so they are also linearly independent over $E$.  It follows that $C\otimes E$ is also an MDS code, cf. \cite[Chapter~11]{MacWilliamsSloane1977}.

 We clearly have $\mathrm{PAut}(C)\le\mathrm{PAut}(C')$. Take $g\in \mathrm{PAut}(C')$, and let $B_g$ be its permutation matrix, cf. Remark \ref{rem_ElePAut}. There is an invertible matrix $P_g$ of order $k$ over $E$ such that $P_gM=MB_g$. Let $M_0$ and $N_0$ be the matrices formed by the first $k$ columns of $M$ and $MB_g$ respectively. Then $M_0$ is invertible by the fact $C$ is an MDS case, and $M_0,N_0$ both have entries in $\F$. It follows that $P_g=N_0M_0^{-1}$ and it has entries in $\F$. We conclude that $g$ stabilizes $C$, i.e., $g\in{\rm PAut}(C)$ as desired. This completes the proof.
 \end{proof}

\begin{lemma}\label{lem_basic-bound}
If $C$ is a nontrivial $[n,k]_q$ MDS code, then $n\le q+\min\{k,n-k\}-1$. In particular, we have $n\le2q-2$.
\end{lemma}
\begin{proof}
We have $n\le q+k-1$ by \cite[Lemma~1.2]{Ball2012}. Since $C^\perp$ is also an MDS code, we have $n\le q+(n-k)-1$ by the same reason. This proves the first claim. Since $\min\{k,n-k\}\le n/2$, it follows that $n\le q-1+n/2$, i.e., $n\le 2q-2$. 
\end{proof}
The following result is an immediate corollary of Lemma \ref{lem_basic-bound}.
\begin{corollary}\label{cor_ResValq}
If there exists a nontrivial $[n,k]_q$ MDS code with $n>q+1$, then $\frac{n+2}{2}\le q\le n-2$.
\end{corollary}

Let $\overline{\F}$ be the algebraic closure of $\F$. In particular, each polynomial of positive degree in $\F[x]$ factorizes into the product of degree $1$ polynomials in $\overline{\F}[x]$. To utilize Lemma \ref{lem_basic-bound} more efficiently, we need the following lemma. We define the Frobenius map
\begin{equation}\label{eqn_sigma}
\sigma:\overline{\F}^\Omega\rightarrow \overline{\F}^\Omega,\quad (x_\omega)_{\omega\in\Omega}\mapsto (x_\omega^p)_{\omega\in\Omega}.
\end{equation}

\begin{lemma}\label{lem_sigCommuG}
Let $U$ be an $\overline{\F} G$-submodule of $\overline{\F}^\Omega$, where $G$ is a subgroup of $\sym(\Omega)$. Then $\sigma(U)$ is also an $\overline{\F} G$-submodule of $\overline{\F}^\Omega$, and $\dim(U)=\dim(\sigma(U))$.
\end{lemma}
\begin{proof}
It is routine to verify that elements of $G$ and $\sigma$ commute. It follows that $\sigma(U)$ is invariant under the action of $G$, i.e., it is an $\overline{\F} G$-module. Since $\sigma$ is a bijection of $\overline{\F}^\Omega$, we clearly have $\dim(U)=\dim(\sigma(U))$. This completes the proof.
\end{proof}

\begin{lemma}\label{lem_FrobRealize}
Suppose that $U$ is a subspace of $\overline{\F}^\Omega$ such that $\sigma^b(U)=U$ for some positive integer $b$. Then $U':=U\cap\F_{p^b}^\Omega$ has the same dimension as $U$, and $U=U'\otimes \overline{\F}$.
\end{lemma}
\begin{proof}
Take a basis $\mathcal{B}$ of $U$, whose coordinates all lie in some finite fields. Let $E$ be the smallest field that contains $\F_{p^b}$ and all those coordinates, which is again a finite field. Take a basis $\gamma_0,\ldots,\gamma_{m-1}$ of $E$ over $\F_{p^b}$. Let $U':=U\cap\F_{p^b}^\Omega$. For each vector $v\in \mathcal{B}$ and $0\le i\le m-1$, define $u_i=\sum_{j=0}^{m-1}\gamma_i^{p^{bj}}\sigma^{bj}(v)\in U$. We have $u_i\in U'$ upon direct check. The matrix $(\gamma_i^{p^{bj}})_{0\le i,j\le m-1}$ is nonsingular by \cite[Corollary~2.38]{LidlNiederreiter1997}, so $v$ lies in the $E$-span of $u_0,\ldots,u_{m-1}$.  It follows that $U\le U'\otimes \overline{\F}$, and $U$ is generated by a set of vectors in $U'$. The two claims then follow immediately. This completes the proof.
\end{proof}

Let $U$ be an $\overline{\F}G$-submodule of $\overline{\F}^\Omega$, and let $m=\dim_{\overline{\F}}U$. We say that $U$ can be realized over $\F_{p^s}$ if $U\cap\F_{p^s}^\Omega$ has dimension $m$, i.e., $U$ has a basis that is contained in $\F_{p^s}^\Omega$. By Lemma \ref{lem_FrobRealize}, we deduce that $U$ can be realized over $\F_{p^s}$ if and only if $\sigma^s(U)=U$.

\begin{lemma}\label{lem_finFieldObs}
Let $U\le\overline{\mathbb{F}}^\Omega$ be a nontrivial $[n,k]$ MDS code. If $\sigma^b(U)=U$ for some positive integer $b$, then $n\le 2p^b-2$.
\end{lemma}
\begin{proof}
Let $U':=U\cap\F_{p^b}^\Omega$, and write $d(U')$ for its minimum weight. By Lemma \ref{lem_FrobRealize}, it has dimension $k$ over $\F_{p^b}$, and $U=U'\otimes \overline{\F}$. By the Singleton bound, we have $d(U')\le n-k+1$. By the fact $U'\subseteq U$, we have $d(U')\ge n-k+1$. It follows that $U'$ is also an $[n,k]$ MDS code. The claim then follows by applying Lemma \ref{lem_basic-bound} to the code $U'$.
\end{proof}

A nonzero $\overline{\F}G$-submodule $U$ of $\overline{\F}^\Omega$ is \textit{irreducible} if its only submodule are 0 and $U$. The \textit{socle} of $\overline{\F}^\Omega$ is the sum of all its irreducible submodules. An irreducible submodule of $C\otimes\overline{\F}$ is an irreducible submodule of $\overline{\F}^\Omega$, so is contained in the socle of $\overline{\F}^\Omega$. If $G$ is $2$-transitive on $\Omega$, the $\overline{\F} G$-module structure of $\overline{\F}^\Omega$ has been extensively studied, cf. Burichenko \cite{Burichenko2000}, Landrock and Michler \cite{LandrockMichler1980}, Mortimer \cite{mortimer1980modular}, Hiss \cite{Hiss1990,Hiss2004}, Guralnick and Tiep \cite{GuralnickTiep2011}. In particular, Mortimer's result \cite{mortimer1980modular} provides the motivation for the research in this paper. Despite of the fact that its title is about known $2$-transitive groups, his result covers all of them modulo the classification of finite simple groups, cf. \cite{Cameron1999}.

\section{Reduction to the reducible hearts}
\label{sec:reduction}

In this and the next section, we present the proof of Theorem \ref{thm:main}. We first fix some notation that we shall use throughout the remaining part of this paper. Let $q=p^f$ with $p$ prime, and let $\F=\F_q$ be the finite field with $q$ elements. We write $\overline{\F}$ for the algebraic closure of $\F$. Let $C$ be a nontrivial $[n,k]_q$ MDS code. Let $\Omega$ be the set of $n$ coordinate positions of $C$, so that $C$ is a subspace of $\F^\Omega$. Let $G=\PAut(C)$, and assume that $G$ acts $2$-transitively on the set $\Omega$. By Proposition \ref{prop:PAutdual} and Lemma \ref{lem_duality}, we assume without loss of generality that $k\le n/2$.

\begin{lemma}\label{lem_Gfinite}
Take notation as above. Let $G_\alpha$ be the stabilizer of some $\alpha\in\Omega$ in $G$. If $(G,G_\alpha,p)$ is one of the cases in Table \ref{tab:remaining-finite}, then we have $n\le q+1$.
\end{lemma}
\begin{proof}
Let $(G,G_\alpha)$ be one of the cases in Table \ref{tab:remaining-finite}. Suppose to the contrary that $n>q+1$. Then by Corollary \ref{cor_ResValq}, we have $\frac{n+2}{2}\le q\le n-2$. This gives the possible values of $q=|\F|$ in the last column of the table. If $G=HS$ and $(n,q)=(176,128)$, we would have $\min\{k,n-k\}\ge n+1-q=49$ if $C$ is a nontrivial MDS code by Lemma \ref{lem_basic-bound}. We check with Magma \cite{BosmaCannonPlayoust1997} that all $G$-invariant codes over $\F_{128}$ of length $176$ and dimension in the range $[2,174]$ have dimensions $21,22,154$ or $155$, none of which satisfies $\min\{k,176-k\}\ge 49$. If $G=Co_3$ and $q=243$, the parameters of all $G$-invariant codes of length $276$ and dimension in the range $[2,274]$ are $[276,126,\le 102]_{243}$, $[276,127,\le 105]_{243}$, $[276,149,\le 84]_{243}$ and $[276,150,\le 86]_{243}$. If $G=Co_3$ and $q=256$, all $G$-invariant codes of length $276$ and dimension in the range $[2,274]$ are $[276,23,\le 132]_{256}$ and $[276,253,\le 14]_{256}$ codes. None of those codes are MDS codes.  For the other cases, we list all the $G$-invariant $[n,k,d]_q$ codes with $2\le k\le n-2$ as follows:
\begin{enumerate}
    \item $G={}^2G_2(3)$: one $[28,7,12]_{16}$ code, two $[28,8,12]_{16}$ codes and one $[28,9,10]_{16}$ code.
    \item $G=M_{24}$: one $[24,12,8]_{16}$ code.
    \item $G=M_{23}$: one $[23,11,8]_{16}$ code and one $[23,12,7]_{16}$ code.
    \item $G=M_{22}$: one $[22,10,8]_{16}$ code, one $[22,11,6]_{16}$ code, two $[22,11,7]_{16}$ code, fourteen $[22,11,8]_{16}$ codes and one $[22,12,6]_{16}$ code.
    \item $G=M_{11}$: one $[12,6,6]_9$ code.
\end{enumerate}
None of the above codes are MDS codes.  This completes the proof.
\end{proof}

\begin{table}[!htbp]
\centering
\footnotesize
\setlength{\tabcolsep}{3pt}
\renewcommand{\arraystretch}{1.4} 
\caption{The $(G,G_\alpha,p)$ tuples. Here, $P\in\mathrm{Syl}_3(\mathrm{PSL}_2(8))$ for line 1, and the possible values of $q=|\F|$ follows from Corollary \ref{cor_ResValq}.}
\label{tab:remaining-finite}
\begin{tabular}{|c|c|c|c|c|}
\hline
$G$ & $G_\alpha$ & $|\Omega|=[G:G_\alpha]$ & $p=\mathrm{char}(\F)$ & Possible $q=|\F|$  \\
\hline
$G={}^2G_2(3)\cong\mathrm{P\Gamma L}_2(8)$ & $G_\alpha=N_G(P)$ & $28$ & $2$ & $16$  \\
\hline
$G=M_{24}$ & $G_\alpha=M_{23}$ & $24$ & $2$& $16$  \\
\hline
$G=M_{23}$ & $G_\alpha=M_{22}$ & $23$ & $2$& $16$  \\
\hline
$G=M_{22}$ & $G_\alpha=\mathrm{PSL}_3(4)$ & $22$ & $2$& $16$ \\
\hline
$G=M_{11}$ & $G_\alpha=\mathrm{PSL}_2(11)$ & $12$ & $3$& $9$\\
\hline
$G=HS$ & $G_\alpha=\mathrm{PSU}_3(5){:}2$ & $176$ & $2$ & $128$\\
\hline
$G=Co_3$ & $G_\alpha=\mathrm{McL}{:}2$ & $276$ & $3$ & $243$\\
       &                               &       & $2$ & $256$\\

\hline
\end{tabular}
\end{table}

By Lemma~\ref{lem_scalar-extension}, the code $C\otimes\overline{\F}$ is also an MDS code. By abuse of notation, we assume that $\F$ is algebraically closed and write $C$ for $C\otimes \overline{\F}$ from now on. Let $M=\overline{\F}^{\Omega}$, so that $C$ is a $\overline{\F}G$-submodule of $M$. We define two subspaces of $M$ as follows:
\begin{equation*}\label{eqn_defJA}
J=\overline{\F}\mathbf{1},\qquad
A=\left\{(a_\omega)_{\omega\in\Omega}:\sum_{\omega\in\Omega}a_\omega=0\right\}.
\end{equation*}
We have $A=J^\perp$, $A\cap J=0$ if $p\nmid n$ and $J\le A$ if $p\mid n$. Both $J$ and $A$ are invariant under the action of $\sym(\Omega)$, so are $\F G$-submodules of $M$. The \textit{heart} of $M$ is the quotient module
\[
  \overline A=A/(A\cap J).
\]
By Lemma~2 of \cite{mortimer1980modular}, we have the following result.
\begin{lemma}\label{lem_standard-subspaces}
Take notation as above. If the heart $\overline A$ is a simple $\F G$-module, then the only $G$-submodules of $M$ are $0$, $J$, $A$ and $M$.
\end{lemma}

The \textit{socle} of $G$, denoted by $\mathrm{soc}(G)$, is the subgroup generated by all the minimal normal subgroups of $G$. By Burnside's theorem \cite[Theorem IX on p. 214]{Burnside1911}, a finite $2$-transitive group is either affine or almost simple, cf. \cite[Table~7.4]{Cameron1999} and \cite[Section~7.7]{DixonMortimer1996}. In the almost simple case, there is a nonabelian simple group $S$ such that $\mathrm{soc}(G)=S$  and $S\le G\le\operatorname{Aut}(S)$. In the affine case, $\mathrm{soc}(G)\cong(\F_\ell^a,+)$ is an elementary abelian $\ell$-group with $\ell$ prime, and it acts regularly on $\Omega$. We regard it as an vector space over $\F_{\ell}$, and denote it by $V$. We identify $\Omega$ with $V$ via $\alpha^v\mapsto v$, where $\alpha$ is a fixed point of $\Omega$ and $v$ ranges in $V$. We then have $G=V\rtimes G_0$, where $G_0\le\mathrm{GL}(V)$ is the stabilizer of $0$ in $G$. The group $G_0$ is transitive on $V\setminus\{0\}$ by the $2$-transitivity assumption.
\begin{proposition}\label{prop:affine-exclusion}
Let $C$ be a nontrivial $[n,k]_q$ MDS code. If $G=\mathrm{PAut}(C)$ is a $2$-transitive subgroup of $\sym(\Omega)$ of affine type, then we have $n\le q$.
\end{proposition}
\begin{proof}
Let $\ell^a$ be the order of $\mathrm{soc}(G)$, where $\ell$ is a prime. We have $C\not\in\{0,J,A,M\}$ by the assumption $\min\{k,n-k\}\ge 2$. By Lemma~\ref{lem_standard-subspaces}, the heart $\overline{A}$ is reducible. By line 4 of \cite[Table~1]{mortimer1980modular}, we deduce that $p=\ell$. Hence, $n=p^a$ and $q=p^f$. Suppose to the contrary that $n\ge q+1$, i.e., $a>f$. By Lemma~\ref{lem_basic-bound}, we have $n\le 2q-2$. It follows that $p^a=n\le2p^f-2<p^{f+1}$, i.e., $a<f+1$: a contradiction. Therefore, we have $n\le q$ as desired.
\end{proof}

\begin{proposition}\label{prop:mortimer-reduction}
Take notation as above, and assume that $G$ is $2$-transitive on the set $\Omega$. If $n=|\Omega|>q+1$, then $(G,\Omega,q)$  occur in Table~\ref{tab:remaining-families}.
\end{proposition}
\begin{proof}
Assume that $n>q+1$. By Proposition \ref{prop:affine-exclusion}, $G$ is an almost simple group, so the case $G=A_4\cong\mathrm{AGL}_1(4)$ in the line 3 of  \cite[Table~1]{mortimer1980modular} is excluded. Let $S=\mathrm{soc}(G)$, which is a nonabelian simple group. By Lemma \ref{lem_standard-subspaces}, the heart $\overline{A}=A/(A\cap J)$ is reducible, so it must be one of the cases in \cite[Table~1]{mortimer1980modular}. By Theorems 61 and 63 of \cite{BallLavrauw2020}, we know that the MDS conjecture holds if $\min\{k,n-k\}\le 5$ and so the case $G=\mathrm{PSL}_2(11)$ in the third row from the bottom of \cite[Table~1]{mortimer1980modular} is ruled out. If $S={}^2G_2(3)'$, then $S\cong\mathrm{PSL}_2(8)$ and $G=\mathrm{P\Gamma L}_2(8)$ by \cite[Table~1]{mortimer1980modular}. This case has been excluded by Lemma \ref{lem_Gfinite}. If $S=\mathrm{Sp}_4(2)'$, then $S\cong\mathrm{PSL}_2(9)$, and it has a unique $2$-transitive action on $10$ elements. Hence this case is covered by the action of $\mathrm{PSL}_2(9)$ on the projective points of $\mathrm{PG}(1,9)$. If $S={}^2G_2(r)$ with $r=3^{2a+1}\ge 27$, the possibility that $p\mid r-\sqrt{3r}+1$ left open in \cite[Table~1]{mortimer1980modular} has been eliminated in \cite[Section 4.3]{GuralnickTiep2011}. We have excluded the cases in Table \ref{tab:remaining-finite} by Lemma \ref{lem_Gfinite}, and the remaining cases in \cite[Table~1]{mortimer1980modular} are recorded in Table \ref{tab:remaining-families}. This completes the proof.
\end{proof}

\begin{table}[!htbp]
\centering
\small
\setlength{\tabcolsep}{4pt}
\renewcommand{\arraystretch}{1.4} 
\caption{The candidate $2$-transitive actions, where $p=\mathrm{char}(\F)$ and $a\ge 1$}
\label{tab:remaining-families}
\begin{tabular}{|c|c|c|c|}
\hline
$S=\mathrm{soc}(G)$ &  $\Omega$ & $n=|\Omega|$ &Conditions \\
\hline
$\mathrm{PSL}_d(r)$ &points of $\mathrm{PG}(d-1,r)$ & $(r^d-1)/(r-1)$ & $d\ge3$, $r=p^s$\\
\hline
$\mathrm{Sp}_{2m}(2)$ &quadratic forms of type $\varepsilon$& $2^{m-1}(2^m+\varepsilon 1)$ & $p=2$, $m\ge3$, $\varepsilon=\pm$  \\
\hline
$\mathrm{PSL}_2(r)$ & points of $\mathrm{PG}(1,r)$ & $r+1$ & $p=2$, $r\ge 5$ odd \\
\hline
$\mathrm{Sz}(r)$ &  Suzuki-Tits ovoid & $r^2+1$ & $r=2^{2a+1}$, $p\mid r+\sqrt{2r}+1$ \\
\hline
$\mathrm{PSU}_3(r)$& points of the classical unital & $r^3+1$& $r\ge 3$, $p\mid r+1$ \\
\hline
${}^2G_2(r)$& Ree--Tits ovoid & $r^3+1$ & $r=3^{2a+1}$,   $p\mid 2(r+1)(r+\sqrt{3r}+1)$ \\
\hline
\end{tabular}
\end{table}

\section{Almost simple actions}
\label{sec:almost-simple}

Our main result in this section is the following result.
\begin{theorem}\label{thm:AS}
Let $C$ be a nontrivial $[n,k]_q$ MDS code such that $C\le\F_q^\Omega$, and assume that $G=\PAut(C)$ is a $2$-transitive subgroup of $\sym(\Omega)$ of almost simple type in Table \ref{tab:remaining-families}. Then $q$ is a power of $4$, $\mathrm{PSL}_2(5)\unlhd G$, and $C$ is a $[6,3,4]_q$ code.
\end{theorem}

In view of Proposition \ref{prop:mortimer-reduction}, Theorem \ref{thm:AS} would complete the proof of Theorem \ref{thm:main}. We take the same notation as introduced in the beginning of Section \ref{sec:reduction}, and further assume that $G=\PAut(C)$ is a 2-transitive permutation group of $\Omega$ that appears in Table \ref{tab:remaining-families}. Let $S=\mathrm{soc}(G)$, which is a nonabelian simple group. We have $S\le G\le\Aut(S)$.

\begin{lemma}\label{lem_uniser}
Suppose that the $\overline{\F} S$-submodules of $M$ form a chain under inclusion. Then $|\Omega|=n\le 2p-2$, where $p$ is the characteristic of $\F$.
\end{lemma}
\begin{proof}
Let $\sigma$ be as defined in \eqref{eqn_sigma}. By Lemma \ref{lem_sigCommuG}, $\sigma(C)$ is also an $\overline{\F} S$-submodule of the same dimension. There is at most one submodule of a given dimension by assumption, so $\sigma(C)=C$. The claim then follows from Lemma \ref{lem_finFieldObs}.
\end{proof}

\subsection{The case $S=\mathrm{PSL}_d(p^s)$, $d\ge 3$}
\label{sec:projective-linear}
Let $r=p^s$ and $d\ge3$, where  $p$ is the characteristic of $\F$ and $s$ is a positive integer. We set $W=\mathbb{F}_r^d$, and let $\Omega$ be the set of all one-dimensional $\F_r$-subspaces of $W$. We have $n = |\Omega| = \frac{r^d-1}{r-1}$, and $S=\mathrm{PSL}_d(r)$. A \textit{hyperplane} of $W$ is an $(d-1)$-dimensional subspace of $W$. For a hyperplane $H$, let $\mathrm{PG}(H)$ be the set of $1$-dimensional subspaces contained in $H$. Let $D$ be the $\overline{\F} S$-submodule of $M=\overline{\F}^\Omega$ spanned by $\mathbf{1}_{\PG(H)}$ with $H$ ranging over all hyperplanes of $W$, which is referred to as \textit{the hyperplane design module} in the literature. We have $J\le D$, since each element of $\Omega$ is contained in $n_0=\frac{r^{d-1}-1}{r-1}$ hyperplanes. We have $\dim(D\cap A)=\dim(D)-1$ and $M=J\oplus A$ by the fact $p\nmid n$. We refer to \cite{BardoeSin2000} for the detailed descriptions of $\overline{\F} S$-module structure of $M=\overline{\F}^\Omega$.
\begin{theorem}\label{thm:projective-linear}
Take notation as above. Then $\mathrm{soc}(G)$ can not be $\mathrm{PSL}_d(p^s)$, $d\ge 3$.
\end{theorem}
\begin{proof}
Suppose to the contrary that $\mathrm{soc}(G)$ is $S=\mathrm{PSL}_d(p^s)$, $d\ge 3$. By replacing $C$ with $C^\perp$, we assume without loss of generality that $k\le n/2$. It is shown in the proof of \cite[Lemma 9.2]{DevillersEtAl2025} that $A\cap D$ is the unique minimal submodule of $A$ by using the results in \cite{BardoeSin2000}. That is, $\mathrm{soc}(A)=D\cap A$. Since $\dim(C)\ge 2$, we deduce that $C\cap A$ contains $D\cap A$.

Let $H_1,H_2$ be two distinct hyperplanes of $W$. We have $\mathbf{1}_{\Omega\setminus \mathrm{PG}(H_1)},\mathbf{1}_{\Omega\setminus \mathrm{PG}(H_2)}\in D\cap A$, since $n-|\mathrm{PG}(H_i)|\equiv 0\pmod{p}$ for $i\in\{1,2\}$. It follows that they are codewords of $C$, and so
\[
\mathbf{1}_{\Omega\setminus \mathrm{PG}(H_1)}-\mathbf{1}_{\Omega\setminus \mathrm{PG}(H_2)}=\mathbf{1}_{\mathrm{PG}(H_2)}-\mathbf{1}_{\mathrm{PG}(H_1)}\in C.
\]
This codeword has weight $2r^{d-2}$, since $\dim_{\F_r}(H_1\cap H_2)=d-2$. It follows that
\[
\frac{n}{2}+1\le n-k+1=d(C)\le2r^{d-2},
\]
i.e., $\frac{r^d-1}{r-1}+2\le4r^{d-2}$. This inequality holds for no $(r,d)$ pairs with $d\ge 3$ upon direct check: a contradiction. This completes the proof.
\end{proof}

\subsection{The case $S=\mathrm{Sp}_{2m}(2)$, $m\ge 3$}
\label{sec:symplectic}
Let $V=\mathbb{F}_2^{2m}$ with $m\ge3$, and let $b$ be the alternating form on $V$ such that $b(x,y)=\sum_{i=1}^{2m}x_iy_{2m+1-i}$ for $x,y\in V$. Let $S=\mathrm{Sp}_{2m}(2)$ be the corresponding symplectic group. For a quadratic form $Q:V\rightarrow\F_2$, its associated bilinear form is $f(x,y)=Q(x+y)-Q(x)-Q(y)$. Since the characteristic is $2$, the bilinear form $f$ is an alternating form, i.e., $f(x,x)=0$ for all $x\in V$. A subspace $U$ of $V$ is \textit{totally singular} if $Q(u)=0$ for all $u\in U$. The Witt index of $Q$ is the maximal dimension of its totally singular subspaces.

Let $\mathcal{X}$ be the set of all quadratic forms on $V$ with associated bilinear form $b$. Let $e_1,\ldots,e_{2m}$ be the standard basis of $V$. For each $Q\in\mathcal{X}$, we have
\begin{equation}\label{eqn_Qadef}
  Q(x)=\sum_{i=1}^{m}x_ix_{2m+1-i}+\sum_{i=1}^{2m}Q(e_i)x_i.
\end{equation}
Here, we use the fact that $a^2=a$ for $a\in\F_2$. For $Q\in \mathcal{X}$, we say that it is hyperbolic or has type $+$ if its Witt index is $m$, and say that it is elliptic or has type $-$ if its Witt index is $m-1$. Let $\Omega^\varepsilon$ be the set of quadratic forms in $\mathcal{X}$ of type $\varepsilon$ for $\varepsilon\in\{+,-\}$. We have $|\Omega^\varepsilon|=2^{m-1}(2^m+\varepsilon 1)$. Since $G=S$, the group $S$ acts $2$-transitively on both sets via $Q\mapsto Q^g$ for $Q\in\mathcal{X}$ and $g\in S$, where $Q^g(x)=Q(x^{g^{-1}})$ for $x\in V$. We refer to \cite{SastrySin2002} for more details.
\begin{lemma}\label{lem_muQ}
For $Q\in\mathcal{X}$, let $\mu(Q)=\sum_{i=1}^{m}Q(e_i)Q(e_{2m+1-i})$. We have $\Omega^+=\{Q\in\mathcal{X}:\mu(Q)=0\}$  and $\Omega^-=\{Q\in\mathcal{X}:\mu(Q)=1\}$.
\end{lemma}
\begin{proof}
Let $a_j=Q(e_j)$ for $1\le j\le 2m$, and fix an integer $i$ such that $1\le i\le m$. The quadratic form $a_ix_i^2+x_ix_{2m+1-i}+a_{2m+1-i}x_{2m+1-i}$ on the subspace $\la e_i,e_{2m+1-i}\ra$ is elliptic if $a_i=a_{2m+1-i}=1$ and is hyperbolic otherwise. That is, its type depends on the value of $a_ia_{2m+1-i}$. The claim then follows by taking summation over $1\le i\le m$.
\end{proof}

Suppose that $\Omega=\Omega^\varepsilon$ with $\varepsilon\in\{+,-\}$. We have $p=2$ in this case by Table \ref{tab:remaining-families}. We extend a quadratic form in $\Omega$ to a quadratic form over $V\otimes\overline{\F}$ in the natural way. We also extend $b$ to a bilinear form over $V\otimes\overline{\F}$ similarly. By \cite[Proposition~8.4]{SastrySin2002} (i), $J$ is the socle of $M=\overline{\F}^{\Omega}$, i.e., the unique irreducible $\overline{\F} S$-submodule of $M$.  By \cite[Proposition~8.4]{SastrySin2002} (ii), $M/J$ has a unique irreducible module $N^\varepsilon$ isomorphic to $V\otimes\overline{\F}$.

We now construct the full preimage of $N^\varepsilon$ in $M$. We define
\[
  Q^+(x)=\sum_{i=1}^{m}x_ix_{2m+1-i},\quad Q^-(x)=x_1+x_{2m}+Q^+(x),
\]
so that $Q^+\in\Omega^+$ and $Q^-\in\Omega^-$. For each $Q\in\Omega$, $Q(x)+Q^\varepsilon(x)$ is linear map from $V\otimes\overline{\F}$ to $\overline{\F}$. For $w\in V$, define
$\lambda_w\in M$ by
\[
  (\lambda_w)_Q=Q(w)+Q^\varepsilon(w),\quad\text{for } Q\in\Omega.
\]
\begin{proposition}\label{prop:symplectic-structure}
Let  $B^\varepsilon=\left\langle \mathbf 1,\lambda_w:w\in V
 \right\rangle_{\overline{\mathbb F}}$.
Then $B^\varepsilon$ is the full preimage of $N^\varepsilon$ in $M=\overline{\F}^\Omega$, and we have $B^\varepsilon\le C\cap C^\perp$.
\end{proposition}
\begin{proof}
First, we claim that $\lambda_{u+v}=\lambda_u+\lambda_v$ for $u,v\in V\otimes\overline{\F}$. Indeed,  we have
\begin{align*}
 \lambda_{u+v}(Q)
 &=Q(u+v)+Q^\varepsilon(u+v)\\
 &=Q(u)+Q(v)+b(u,v)+Q^\varepsilon(u)+Q^\varepsilon(v)+b(u,v)\\
 &=\lambda_u(Q)+\lambda_v(Q)
\end{align*}
for $Q\in\Omega$.  For $g\in S$, we have $\lambda_w^g=\lambda_{w^g}+\left(Q^\varepsilon(w^g)+Q^\varepsilon(w)\right)\mathbf{1}\in B^\varepsilon$, since
\begin{align*}
 (\lambda_w^g)(Q)
 =\lambda_w(Q^{g^{-1}})
 =Q(w^g)+Q^\varepsilon(w)
 =\lambda_{w^g}(Q)+Q^\varepsilon(w^g)+Q^\varepsilon(w)
\end{align*}
for $Q\in\Omega$. It follows that $B^\varepsilon$ is a $\overline{\F} S$-submodule of $M$.

We claim that $B^\varepsilon\ne J$.  Take $Q\in\Omega$ with $Q\ne Q^\varepsilon$. Since $Q(x)+Q^\varepsilon(x)$ is a nonzero linear map from $V\otimes\overline{\F}$ to $\overline{\F}$, there is $w\in V\otimes\overline{\F}\setminus\{0\}$ such that $Q(w)+Q^\varepsilon(w)\ne 0$, i.e., $\lambda_w(Q)\ne 0$. On the other hand, we have $\lambda_w(Q^\varepsilon)=0$. Hence $\lambda_w\not\in J$. This proves the claim.

By \cite[Proposition~8.4]{SastrySin2002}, we deduce that $B^\varepsilon/J$ contains $N^\varepsilon$. They are equal by considering dimensions. Since an irreducible submodule of $C$ is also an irreducible submodule of $M$, we deduce that $C$ contains $J$. Since $\dim(C)\ge 2$, we have $C\ne J$, and so $C/J$ contains $N^\varepsilon$. We conclude that $C$ contains $B^\varepsilon$. The same arguments yield that $C^\perp$ contains $B^\varepsilon$.  
\end{proof}

Let $\mu=\mu(Q^\varepsilon)$, so that $\mu(Q)=\mu$ for $Q\in\Omega$, cf. Lemma \ref{lem_muQ}. For $Q\in\Omega$, let $v_Q=(Q(e_1),\ldots,Q(e_{2m}))\in\F_2^{2m}$ and $T^\varepsilon=\{v_Q:Q\in\Omega\}$. We have a bijection between $\Omega$ and $T$ via $Q\mapsto v_Q$, cf. \eqref{eqn_Qadef}. We have $T^\varepsilon=\{v\in\F_2^{2m}:Q^+(v)=\mu\}$ by Lemma \ref{lem_muQ} and the fact $\mu(Q)=Q^+(v_Q)$. In particular, we have
\begin{align}
  \Omega=\Omega^\varepsilon&=\left\{\sum_{i=1}^{m}x_ix_{2m+1-i}+\sum_{i=1}^{2m}v_{2m+1-i}x_{i}: v\in T^\varepsilon\right\}\notag\\
     &=\{Q^+(x)+b(v,x): v\in\F_2^{2m}\text{ such that }Q^+(v)=\mu\}.\label{eqn_OmegaEps}
\end{align}
For a property $\mathbf{P}$, let $[\![\mathbf{P}]\!]=1$ if it holds and $[\![\mathbf{P}]\!]=0$ otherwise.

\begin{theorem}\label{thm:symplectic}
Take notation as above. Then $\mathrm{soc}(G)$ can not be $\Sp_{2m}(2)$, $m\ge 3$.
\end{theorem}
\begin{proof}
We continue with the above arguments. By Proposition \ref{prop:symplectic-structure}, we have $B^\varepsilon\le C\cap C^\perp$. We take a nonzero vector $w\in \F_2^{2m}$ such that $Q^+(w)=0$, and take $c=Q^+(w)+Q^\varepsilon(w)\in\F_2$. By the paragraph preceding this theorem, we deduce that $(\lambda_w)_Q=c+b(v,w)\in\F_2$ for $Q\in\Omega$ such that $Q(x)=Q^+(x)+b(v,x)$ with $Q^+(v)=\mu$. The codeword $\lambda_w+c\mathbf 1$ is in $\F_2^\Omega$ and the number of its zero coordinates is
\begin{align*}
 &\sum_{v\in\F_2^{2m}}[\![Q^+(v)=\mu]\!]\cdot [\![b(v,w)=0]\!]
=\frac{1}{4}\sum_{v\in\F_2^{2m}}\left(1+(-1)^{Q^+(v)+\mu}\right)
  \left(1+(-1)^{b(v,w)}\right)\\
&=2^{2m-2}+\frac{1}{4}\sum_{v\in\F_2^{2m}}(-1)^{Q^+(v)+\mu}
  +\frac{1}{4}\sum_{v\in\F_2^{2m}}(-1)^{Q^+(v)+\mu+b(v,w)}.
\end{align*}
Here, we used the fact $\frac{1}{4}\sum_{v\in\F_2^{2m}}(-1)^{b(v,w)}=0$. Since $Q^+(x)=x_1x_{2m}+\cdots+x_mx_{m+1}$, it is elementary to deduce that the above number is $2^{2m-2}+(-1)^\mu 2^{m-1}$. Since the all-one vector is in $C\cap C^\perp$ and has Hamming weight $n=2^{2m-1}+(-1)^\mu 2^{m-1}$, we deduce that either $\lambda_w+c\mathbf 1$ or $\lambda_w+(c+1)\mathbf 1$ has weight less than $\frac{1}{2}n$. It follows that
\[
  n+2=d(C)+d(C^\perp)<\frac{1}{2}n+\frac{1}{2}n=n.
\]
This contradiction completes the proof.
\end{proof}

\subsection{The case $S=\mathrm{PSL}_2(r)$, $r$ odd}
\label{sec:projective-line}
The group $\mathrm{PSL}_2(3)$ is solvable, so we have $r\ge 5$ in this case. Let $r=\ell^s\ge5$ with $\ell$ an odd prime, and let $\Omega$ be the set of $1$-dimensional subspaces of $V=\F_r^2$. By Table \ref{tab:remaining-families}, the characteristic of $\F$ is $p=2$. We identify $\Omega$ with $\F_r\cup\{\infty\}$ via $\la (1,a)\ra\mapsto a$ for $a\in\F_r$ and $\la (0,1)\ra\mapsto\infty$. By  \cite[Lemma 4.1 (i)]{GuralnickTiep2011}, the socle of $M=\overline{\F}^\Omega$ is $J$, i.e., $J$ is the unique irreducible submodule of $M$.

Let $\Delta$ be the set of nonzero squares of $\F_r$, and let $u$ be a root of $X^2+X+1=0$ in $\overline{\F}$. By pp. 14-15 of \cite{mortimer1980modular}, the heart $\overline{A}=A/J$ is the direct sum of two irreducible $\overline{\F} S$-submodules $N_1\oplus N_2$, and $N_1,N_2$ are the only proper submodules of $\overline{A}$. For $i=1,2$, the module $N_i$ is generated by an element $a_ie_0+b_ie_\infty+\mathbf{1}_\Delta$, where the $a_i$'s and $b_i$'s are as in the following table.
\begin{table}[h]
\begin{tabular}{c|cccc|c}
$r\pmod{8}$ & $a_1$ & $b_1$ & $a_2$ & $b_2$ & realized over \\ \hline
$1$         & $0$   & $0$   & $1$   & $1$   & $\F_2$        \\
$3$         & $u$   & $u^2$   & $u^2$   & $u$   & $\F_4$   \\
$5$         & $u$   & $u$   & $u^2$   & $u^2$   & $\F_4$     \\
$7$         & $0$   & $1$   & $1$   & $0$   & $\F_2$
\end{tabular}
\end{table}
Let $C_1,C_2$ be the full preimage of $N_1,N_2$ in $M$ respectively. They are the only $\overline{\F} S$-submodules of $A$ that contains $J$. Moreover, they are orthogonal to each other when $r\equiv 1\pmod{4}$, and are both self-dual when $r\equiv 3\pmod{4}$ by \cite{mortimer1980modular}. Since $\dim(C_i)=\frac{r+1}{2}$ for $i=1,2$, we deduce that $C_1=C_2^\perp$ when $r\equiv 1\pmod{4}$.
\begin{theorem}\label{thm:psl2}
Take notation as above. If $S=\mathrm{PSL}_2(r)$ with $r\ge 5$ odd, then $r=5$ and $C$ is a $[6,3,4]$ code.
\end{theorem}
\begin{proof}
Let $q_0=2^{f_0}$, where $f_0=1$ if $r\equiv \pm 1\pmod{8}$ and $f_0=2$ if $r\equiv \pm 3\pmod{8}$. Let $\sigma$ be as defined in \eqref{eqn_sigma}. By considering dimensions, we deduce that both $C$ and $C^\perp$ contains $J$. It follows that $J\le C\le J^\perp$. Since $\dim(C)\ge 2$, it is either $C_1$ or $C_2$ by the paragraph preceding this theorem. It follows that $C$ can be realized over $\F_{q_0}$, i.e., it has a basis consisting of elements in $\F_{q_0}^\Omega$. We thus have $\sigma^{f_0}(C)=C$. By Lemma \ref{lem_finFieldObs}, we have $n=r+1\le 2q_0-2$. Since $r\ge 5$, we deduce that $r=5$, $q_0=4$. It follows that $k=\dim(C)=\frac{r+1}{2}=3$. There is a unique $[6,3,4]_4$ code up to equivalence, i.e., the hexacode. This completes the proof.
\end{proof}

\subsection{The case $S=\mathrm{Sz}(2^{2a+1})$}\label{sec:suzuki}
Let $r=2^{2a+1}\ge8$ and $m=2^{a+1}$. In this case, we have $|\Omega|=r^2+1$ and $p\mid r+m+1$ by Table \ref{tab:remaining-families}, where $p$ is the characteristic of $\F$. The $\overline{\F} S$-submodules of $M=\overline{\F}^\Omega$ are described in Appendix~D.1 of \cite{Hiss1990}; see also \cite[Section~4.2]{GuralnickTiep2011}. They form a chain under inclusion: $0<J<B_1<B_2<A<M$. Here, $B_1/J$, $B_2/B_1$ and $A/B_2$ are simple modules of dimensions $\frac{m}{2}(r-1)$, $(r-1)(r+1-m)$ and $\frac{m}{2}(r-1)$ respectively.
\begin{theorem}\label{thm:suzuki}
Take notation as above. Then $\mathrm{soc}(G)$ can not be $S=\mathrm{Sz}(2^{2a+1})$, $a\ge 1$.
\end{theorem}
\begin{proof}
We continue with the above arguments. By Lemma \ref{lem_uniser}, we deduce that $n=r^2+1\le 2p-2$. On the other hand, we have $2m=2^{a+2}\le r=2^{2a+1}$ and $p\mid r+m+1$, so $2p-2\le2r+2m\le3r<r^2+1$: a contradiction. This completes the proof.
\end{proof}

\subsection{The case $S=\mathrm{PSU}_3(r)$, $r>2$}
\label{sec:unitary}

Let $r=\ell^s > 2$, where $\ell$ is a prime. Let $W = \mathbb{F}_{r^2}^3$, and define $H(x)=x_1x_3^{r}+x_2^{r+1}+x_3x_1^{r}$. We have $\Omega=\{\la u\ra_{\F_{r^2}}:u\in W\setminus\{0\}, H(u)=0\}$, $|\Omega|=r^3+1$. If $\la u\ra_{\F_{r^2}},\la v\ra_{\F_{r^2}}$ are two distinct elements in $\Omega$, then $\{\la au+bv\ra_{\F_{r^2}}:a,b\in\F_{r^2}^2\setminus\{(0,0)\}\}\cap\Omega$ has size $r+1$, and we call it a block of $\Omega$. Let $\Delta$ be the set of all such blocks. The pair $(\Omega,\Delta)$ with the inclusion relation defines a 2-$(r^3+1,r+1,1)$ design, i.e., the classical unital. We refer to \cite[p.~15]{mortimer1980modular} for more details.

We have $p\mid r+1$ by Table \ref{tab:remaining-families}, where $p$ is the characteristic of $\F$. Let $D=\la \mathbf{1}_B:B\in\Delta\ra_{\overline{\F}}\le M=\overline{\F}^\Omega$, and we call $D$ the design module. The dimension of $D$ is $r(r^2-r+1)$ by \cite[Corollary 5.1]{Hiss2004}, so $\dim(D^\perp)=r^2-r+1$. We have $J\le D, D^\perp\le A=J^\perp$, and $D/J$ is not isomorphic to $D^\perp/J$ by considering dimensions.
\begin{lemma}\label{lem_unitalD}
The modules $D$ and $D^\perp$ are not MDS codes.
\end{lemma}
\begin{proof}
It suffices to consider $D$ by Lemma \ref{lem_duality}. The length is $n=r^3+1$. The dimension of $D$ is $k=r(r^2-r+1)$ by \cite[Corollary 5.1]{Hiss2004}. For a block $B\in\Delta$, $\mathbf{1}_B$ is in $D$ and has weight $r+1$. It follows that $n-k+1=d(D)\le r+1$, i.e., $(r-1)^2\le 0$, which contradicts $r>2$. 
\end{proof}

By \cite[Theorem~4.1]{Hiss2004} and \cite[Lemma~4.1 and Section~4.4]{GuralnickTiep2011}, we have the following information on the $\overline{\F} S$-module structure of $M=\overline{\F}^\Omega$:
\begin{enumerate}
\item[(1)] If $p$ is odd and $p\ne3$, or if $p=2$ and $r\equiv3\pmod4$, then all the proper $\overline{\F} S$-submodules form the chain $J<D^\perp<D<A$.
\item[(2)] If $p=2$ and $r\equiv1\pmod4$, then $D/J$ and $D^\perp/J$ are nonisomorphic irreducible modules and their direct sum is $\overline{A}=A/J$.
\item[(3)] If $p=3$, then $J<D^\perp<D<A$, and $D/D^\perp$ is an irreducible $\overline{\F} \mathrm{PGU}_3(r)$-module which splits into the direct sum of $3$ irreducible $\overline{\F} S$-submodules.
\end{enumerate}

\begin{theorem}\label{thm:unitary}
Take notation as above. Then $\mathrm{soc}(G)$ can not be $S=\mathrm{PSU}_3(r)$, $r>2$.
\end{theorem}
\begin{proof}
We continue with the above arguments. For the cases (1) and (2) listed preceding this theorem,  $D$ and $D^\perp$ are all the proper $\overline{\F} S$-submodules of $M$ of dimension larger than $1$, which we have excluded in Lemma \ref{lem_unitalD}. It remains to consider the case $p=3$. We have $D^\perp<C, C^\perp<D$ by Lemma \ref{lem_unitalD}. By the information about the decomposition matrix of the irreducible complex character $\chi_{q^3}$ (corresponding to $A$) in \cite{OkuDec}, we deduce that $D/D^\perp=V_1\oplus V_2\oplus V_3$, where the $V_i$'s are pairwise nonisomorphic irreducible $\overline{\F} S$-submodules of dimension $m=(r-1)(r^2-r+1)/3$. The fact the $V_i$'s are nonisomorphic $\overline{\F} S$-modules also follows from \cite[Lemma 3.14]{KleshchevTiep2009}.
It follows that $C/D^\perp$, $C^\perp/D^\perp$ are the direct sums of one or two of the $V_i$'s, and so $k=\dim(C)=r^2-r+1+im$ for some $i\in\{1,2\}$. By considering $C^\perp$ instead of $C$ if necessary, we assume $i=1$ without loss of generality.

Let $\sigma$ be as defined in \eqref{eqn_sigma} with $p=3$. Then $\sigma$ fixes $\mathbf{1}_B$ for each block $B$, and so it stabilizes both $D$ and $D^\perp$. There are exactly three candidate modules for $C$, and $\sigma$ permutes them by Lemma \ref{lem_sigCommuG}. It follows that either $\sigma^2(C)=C$ or $\sigma^3(C)=C$. By Lemma \ref{lem_finFieldObs}, we deduce that $n=r^3+1\le 2\cdot3^3-2=52$. On the other hand, we have $r>3$ since $3\mid r+1$ by Table \ref{tab:remaining-families}, a contradiction.
\end{proof}

\subsection{The case $S={}^2G_2(3^{2a+1})$, $a\ge1$}\label{sec:ree}

Let $r=3^{2a+1}$ with $a\ge 1$, and set $m=3^{a+1}$. In this case, we have $n=|\Omega|=r^3+1$, $S={}^2G_2(3^{2a+1})$. The $\overline{\F} S$-submodule structure of $M=\overline{\F}^\Omega$ are available in Appendix~D.2 of \cite{Hiss1990} and \cite[Proposition~3.8]{LandrockMichler1980}; see also \cite[Section~4.3]{GuralnickTiep2011}.
\begin{lemma}\label{lem_Reep2}
We have $p=2$.
\end{lemma}
\begin{proof}
By Table \ref{tab:remaining-families}, either $p$ is odd and $p$ divides $r+1$ or $r+m+1$, or $p=2$. Suppose to the contrary that $p$ is odd. By \cite[Section 4.3]{GuralnickTiep2011}, the $\overline{\F} S$-submodules of $M$ form a chain under inclusion. By Lemma \ref{lem_uniser}, we deduce that $r^3+1\le 2p-2$. It follows that $3^{6a+3}+1\le 2\cdot 3^{2a+1}+2\cdot 3^{a+1}$, which never holds for $a\ge 1$. This completes the proof.
\end{proof}

By \cite[Section~4.3]{GuralnickTiep2011}, there is a $\overline{\F} S$-submodule $B$ of $M$ such that $J<B<B^\perp<A$, where $J=\mathrm{soc}(M)$ and $B/J=\mathrm{soc}(M/J)$. Both $J$ and $B/J$ are irreducible. Moreover, there are pairwise nonisomorphic irreducible $\overline{\F} S$-modules $U_1,U_2,U_1^*$ such that $B^\perp/B\cong U_1\oplus U_2\oplus U_1^*$, and $\dim(U_2)<\dim(U_1)=\dim(U_1^*)$.

\begin{theorem}\label{thm:ree}
Take notation as above. Then $\mathrm{soc}(G)$ can not be $S={}^2G_2(3^{2a+1})$, $a\ge1$.
\end{theorem}
\begin{proof}
We continue with the above arguments. By Lemma \ref{lem_Reep2}, we have $p=2$. Since $J=\operatorname{soc}(M)$ and $B/J=\operatorname{soc}(M/J)$ are irreducible, both $C$ and $C^\perp$ contain $B$. Consequently, $B\le C\le B^\perp$ and $B\le C^\perp\le B^\perp$. Let $\sigma$ be as defined in \eqref{eqn_sigma} with $p=2$. By Lemma \ref{lem_sigCommuG}, $\sigma(B)$, $\sigma(C)$ and $\sigma(B^\perp)$ are also $\overline{\F} S$-submodules of $M$. By comparing dimensions, we deduce that $\sigma(B)=B$ and $\sigma(B^\perp)=B^\perp$. If $C=B$ or $B^\perp$, then $\sigma(C)=C$ by considering dimensions. By Lemma \ref{lem_finFieldObs} we deduce that that $n=r^3+1\le 2p-2=2$: a contradiction. Therefore, we have $B<C<B^\perp$. We similarly have $B<C^\perp<B^\perp$. By replacing $C$ with $C^\perp$, we assume without loss of generality that $\dim(C)\le\dim(C^\perp)$.
Since $B^\perp/B\cong U_1\oplus U_2\oplus U_1^*$ and the three direct summands are
nonisomorphic, we deduce that $C/B\in\{U_1,U_2,U_1^*\}$. By comparing dimensions, we deduce that either $C=\sigma(C)$, or $C=\sigma^2(C)$. By Lemma \ref{lem_finFieldObs}, we deduce that $r^3+1\le2\cdot2^2-2=6$: a contradiction. This completes the proof.
\end{proof}

We have  examined the six cases in Table \ref{tab:remaining-families} in Theorems \ref{thm:projective-linear}, \ref{thm:symplectic}, \ref{thm:psl2}, \ref{thm:suzuki}, \ref{thm:unitary} and \ref{thm:ree} respectively, and this completes the proof of Theorem~\ref{thm:AS}. Together with Proposition \ref{prop:mortimer-reduction}, this completes the proof of Theorem \ref{thm:main}.

\section{Conclusions}
\label{sec:concluding}

In this paper, we verify the MDS conjecture for $[n,k]_q$ codes $C\le\F_q^\Omega$ with a $2$-transitive permutation automorphism group $G=\mathrm{PAut}(C)\le\sym(\Omega)$. When $(G,\Omega)$ are the infinite families of almost simple type in Table \ref{tab:remaining-families}, then $C$ is equivalent to the scalar extension of the hexacode from $\F_4$ to $\F_q$ in Example \ref{ex:two-transitive-mds-codes} (4). The codes in (2) and (3) of Example \ref{ex:two-transitive-mds-codes} admit $2$-transitive permutation group $G$ of affine type. The classification theorems in Section \ref{sec:almost-simple} heavily rely on the detailed information about the modular representations of the respective nonabelian simple groups. It is conceivably more challenging to determine all the MDS codes with a $2$-transitive permutation group $G$ of affine type, since such information is not available in general. We leave it an open problem to determine all the MDS codes with a $2$-transitive permutation group.

Our main theorem, Theorem \ref{thm:main}, initiates the study of MDS codes with a highly transitive permutation group.  It will be of great theoretical interest to verify the MDS conjecture for linear codes that satisfy weaker symmetry hypotheses, particularly for cyclic codes. Further progress in this direction will require new insights and new techniques to overcome our heavy reliance on deep results in modular representation theory.\medskip

\noindent\textbf{Acknowledgments.} The authors used large language model to help with the understanding of modular representations of finite $2$-transitive groups, and they are responsible for the correctness of all the results in this paper. This research was supported by the National Key R\&D Program of China under grant number 2025YFA1017700 and NSFC grants No. 12225110. Haihua Deng is supported by the National Natural Science Foundation of China under Grant No.~123B2011, and the Postdoctoral Fellowship Program and China Postdoctoral Science Foundation under Grant No.~BX20250059. The research of Andrey V. Vasil'ev was carried out within the framework of the Sobolev Institute of Mathematics project FWNF-2026-0017.


\begin{thebibliography}{99}


\bibitem{Ball2012}
S. Ball,
\emph{On sets of vectors of a finite vector space in which every subset of basis size is a basis},
J. Eur. Math. Soc. (JEMS) \textbf{14} (2012), no. 3, 733--748.


\bibitem{BallLavrauw2020}
S. Ball and M. Lavrauw,
\emph{Arcs in finite projective spaces},
EMS Surv. Math. Sci. \textbf{6} (2019), no. 1--2, 133--172.


\bibitem{BardoeSin2000}
M. Bardoe and P. Sin,
\emph{The permutation modules for \(\mathrm{GL}(n+1,\mathbb{F}_q)\) acting on \(\mathbb{P}^n(\mathbb{F}_q)\) and \(\mathbb{F}_q^{n+1}\)},
J. London Math. Soc. (2) \textbf{61} (2000), no. 1, 58--80.


\bibitem{Berger1993}
T. P. Berger,
\emph{Groupes de permutations des codes {MDS} affine-invariants},
Comm. Algebra \textbf{21} (1993), no. 1, 239--254.


\bibitem{BienertKlopsch2010}
R. Bienert and B. Klopsch,
\emph{Automorphism groups of cyclic codes},
J. Algebraic Combin. \textbf{31} (2010), no. 1, 33--52.


\bibitem{BlokhuisBruenThas1988}
A. Blokhuis, A. A. Bruen, and J. A. Thas,
\emph{On {M.D.S.} codes, arcs in \(\mathrm{PG}(n,q)\) with \(q\) even, and a solution of three fundamental problems of {B. Segre}},
Invent. Math. \textbf{92} (1988), no. 3, 441--459.


\bibitem{BlokhuisBruenThas1990}
A. Blokhuis, A. A. Bruen, and J. A. Thas,
\emph{Arcs in \(\mathrm{PG}(n,q)\), {MDS}-codes and three fundamental problems of {B. Segre}: some extensions},
Geom. Dedicata \textbf{35} (1990), 1--11.


\bibitem{BosmaCannonPlayoust1997}
W. Bosma, J. J. Cannon, and C. Playoust,
\emph{The {Magma} algebra system. {I}. The user language},
J. Symbolic Comput. \textbf{24} (1997), no. 3--4, 235--265.


\bibitem{Burichenko2000}
V. P. Burichenko,
\emph{Extensions of abelian $2$-groups by means of $\mathrm{L}_2(q)$ with irreducible action},
Algebra and Logic \textbf{39} (2000), no. 3, 160--183.


\bibitem{Burnside1911}
W. Burnside, \emph{Theory of Groups of Finite Order}, 2nd ed., Cambridge University Press, Cambridge, 1911.


\bibitem{Bush1952}
K. A. Bush,
\emph{Orthogonal arrays of index unity},
Ann. Math. Statist. \textbf{23} (1952), no. 3, 426--434.


\bibitem{Cameron1999}
P. J. Cameron,
\emph{Permutation Groups},
London Mathematical Society Student Texts, vol. 45,
Cambridge University Press, Cambridge, 1999.


\bibitem{ChigiraHaradaKitazume2014}
N. Chigira, M. Harada, and M. Kitazume,
\emph{On the classification of extremal doubly even self-dual codes with \(2\)-transitive automorphism groups},
Des. Codes Cryptogr. \textbf{73} (2014), no. 1, 33--35.


\bibitem{DevillersEtAl2025}
A. Devillers, M. Giudici, D. R. Hawtin, L. Klawuhn, and L. Morgan,
\emph{Linear dimension of group actions},
arXiv:2512.16079 [math.GR], 2025.


\bibitem{DixonMortimer1996}
J. D. Dixon and B. Mortimer,
\emph{Permutation Groups},
Graduate Texts in Mathematics, vol. 163,
Springer-Verlag, New York, 1996.


\bibitem{FengHollmannLiXiang2026}
T. Feng, H. D. L. Hollmann, W. Li, and Q. Xiang,
\emph{The permutation automorphism groups of irreducible cyclic codes},
arXiv:2603.01904 [math.CO], 2026.


\bibitem{GeiselhartEtAl2021}
M. Geiselhart, A. Elkelesh, M. Ebada, S. Cammerer, and S. ten Brink,
\emph{Automorphism ensemble decoding of Reed--Muller codes},
IEEE Trans. Commun. \textbf{69} (2021), no. 10, 6424--6438.


\bibitem{GuendaGulliver2013}
K. Guenda and T. A. Gulliver,
\emph{On the permutation groups of cyclic codes},
J. Algebraic Combin. \textbf{38} (2013), no. 1, 197--208.


\bibitem{GuralnickTiep2011}
R. M. Guralnick and P. H. Tiep,
\emph{First cohomology groups of Chevalley groups in cross characteristic},
Ann. of Math. (2) \textbf{174} (2011), no. 1, 543--559.


\bibitem{HirschfeldKorchmaros1996}
J. W. P. Hirschfeld and G. Korchm{\'a}ros,
\emph{On the embedding of an arc into a conic in a finite plane},
Finite Fields Appl. \textbf{2} (1996), 274--292.


\bibitem{HirschfeldKorchmaros1998}
J. W. P. Hirschfeld and G. Korchm{\'a}ros,
\emph{On the number of rational points on an algebraic curve over a finite field},
Bull. Belg. Math. Soc. Simon Stevin \textbf{5} (1998), 313--340.


\bibitem{HirschfeldThas1991}
J. W. P. Hirschfeld and J. A. Thas,
\emph{General Galois Geometries},
Clarendon Press, Oxford, 1991.


\bibitem{Hiss1990}
G. Hiss,
\emph{Zerlegungszahlen endlicher Gruppen vom Lie-Typ in nicht-definierender Charakteristik},
Habilitationsschrift, RWTH Aachen, 1990.


\bibitem{Hiss2004}
G. Hiss,
\emph{Hermitian function fields, classical unitals, and representations of \(3\)-dimensional unitary groups},
Indag. Math. (N.S.) \textbf{15} (2004), no. 2, 223--243.


\bibitem{Huffman1998}
W. C. Huffman,
\emph{Codes and groups},
in \emph{Handbook of Coding Theory}, vol. 2, edited by V. S. Pless, W. C. Huffman, and R. A. Brualdi, pp. 1345--1440,
North-Holland, Amsterdam, 1998.


\bibitem{HuffmanPless2003}
W. C. Huffman and V. Pless,
\emph{Fundamentals of Error-Correcting Codes},
Cambridge University Press, Cambridge, 2003.


\bibitem{KaufmanLub}
T. Kaufman and A. Lubotzky,
\emph{Edge transitive {R}amanujan graphs and symmetric {LDPC} good
              codes},
in \emph{S{TOC}'12---{P}roceedings of the 2012 {ACM} {S}ymposium on
              {T}heory of {C}omputing}, pp. 359--366,
ACM, New York, 2012.


\bibitem{KleshchevTiep2009}
A. S. Kleshchev and P. H. Tiep,
\emph{Representations of finite special linear groups in
non-defining characteristic},
Adv. Math. \textbf{220} (2009), no.~2, 478--504.


\bibitem{KnappRodrigues2021}
W. D. Knapp and B. G. Rodrigues,
\emph{A useful tool for constructing linear codes},
J. Algebra \textbf{585} (2021), 422--446.


\bibitem{KnappSchmid1980}
W. Knapp and P. Schmid,
\emph{Codes with prescribed permutation group},
J. Algebra \textbf{67} (1980), no. 2, 415--435.


\bibitem{LandrockMichler1980}
P. Landrock and G. O. Michler,
\emph{Principal \(2\)-blocks of the simple groups of Ree type},
Trans. Amer. Math. Soc. \textbf{260} (1980), no. 1, 83--111.


\bibitem{LidlNiederreiter1997}
R. Lidl and H. Niederreiter,
\emph{Finite Fields},
2nd ed., Encyclopedia of Mathematics and its Applications, vol. 20,
Cambridge University Press, Cambridge, 1997.


\bibitem{MaYan2025}
J. Ma and G. Yan,
\emph{On automorphism groups of binary cyclic codes},
Des. Codes Cryptogr. \textbf{93} (2025), no. 5, 1271--1282.


\bibitem{MacWilliams1964}
F. J. MacWilliams,
\emph{Permutation decoding of systematic codes},
Bell System Tech. J. \textbf{43} (1964), no. 1, 485--505.


\bibitem{MacWilliamsSloane1977}
F. J. MacWilliams and N. J. A. Sloane,
\emph{The Theory of Error-Correcting Codes},
North-Holland, Amsterdam, 1977.


\bibitem{MalevichWillems2014}
A. Malevich and W. Willems,
\emph{On the classification of the extremal self-dual codes over small fields with \(2\)-transitive automorphism groups},
Des. Codes Cryptogr. \textbf{70} (2014), no. 1--2, 69--76.


\bibitem{mortimer1980modular}
B. Mortimer,
\emph{The modular permutation representations of the known doubly transitive groups},
Proc. London Math. Soc. (3) \textbf{41} (1980), no. 1, 1--20.


\bibitem{OkuDec}
T. Okuyama and K. Waki, \emph{Decomposition numbers of ${\rm SU}(3,q^2)$}, J. Algebra {\bf 255} (2002), no.~2, 258--270.


\bibitem{PaceSonnino2017}
N. Pace and A. Sonnino,
\emph{On linear codes admitting large automorphism groups},
Des. Codes Cryptogr. \textbf{83} (2017), no. 1, 115--143.


\bibitem{SastrySin2002}
N. S. N. Sastry and P. Sin,
\emph{On the doubly transitive permutation representations of \(\mathrm{Sp}(2n,\mathbb{F}_2)\)},
J. Algebra \textbf{257} (2002), no. 2, 509--527.


\bibitem{Segre1955Curve}
B. Segre,
\emph{Curve razionali normali e \(k\)-archi negli spazi finiti},
Ann. Mat. Pura Appl. \textbf{39} (1955), 357--379.


\bibitem{Singleton1964}
R. C. Singleton,
\emph{Maximum distance \(q\)-nary codes},
IEEE Trans. Inform. Theory \textbf{10} (1964), 116--118.


\bibitem{StormeThas1993}
L. Storme and J. A. Thas,
\emph{{M.D.S.} codes and arcs in \(\mathrm{PG}(n,q)\) with \(q\) even: an improvement of the bounds of Bruen, Thas, and Blokhuis},
J. Combin. Theory Ser. A \textbf{62} (1993), no. 1, 139--154.


\bibitem{Voloch1991}
J. F. Voloch,
\emph{Complete arcs in Galois planes of non-square order},
in \emph{Advances in Finite Geometries and Designs}, edited by J. W. P. Hirschfeld, D. R. Hughes, and J. A. Thas, pp. 401--406,
Oxford University Press, Oxford, 1991.


\end{thebibliography}
\end{document}